\documentclass[oneside,12pt]{amsart}

\usepackage{amsfonts}
\usepackage{amssymb}
\usepackage{amsxtra}
\usepackage{amstext}
\usepackage[english]{babel}
\usepackage[T1]{fontenc}
\usepackage{latexsym}
\usepackage[latin1]{inputenc}

\newtheorem{theorem}{Theorem}[section]
\newtheorem{lemma}[theorem]{Lemma}
\newtheorem{corollary}[theorem]{Corollary}
\newtheorem{proposition}[theorem]{Proposition}
\newtheorem{example}[theorem]{Example}
\newtheorem{remark}[theorem]{Remark}

\def\bit{\begin{itemize}}
\def\eit{\end{itemize}}
\reversemarginpar   
\def\bc{\begin{center}}
\def\ec{\end{center}}
\def\bthm{\begin{theorem}}
\def\ethm{\end{theorem}}
\def\bcor{\begin{corollary}}
\def\ecor{\end{corollary}}
\def\bprop{\begin{proposition}}
\def\eprop{\end{proposition}}
\def\blem{\begin{lemma}}
\def\elem{\end{lemma}}

\def\brem{\begin{remark}}
\def\erem{\end{remark}}

\def\bdes{\begin{description}}
\def\edes{\end{description}}

\def\iti{\item[(i)]}
\def\itii{\item[(ii)]}

\def\beq{\begin{equation}}
\def\eeq{\end{equation}}
\def\ben{\begin{enumerate}}
\def\een{\end{enumerate}}
\def\beqar{\begin{eqnarray}}
\def\eeqar{\end{eqnarray}}
\def\beqarr{\begin{eqnarray*}}
\def\eeqarr{\end{eqnarray*}}

\def\RR{{\mathbb R}}  

\def\EE{{\mathbb E}}
\def\cA{\mathcal{A}}  
 \def\cE{\mathcal{E}} \def\cF{\mathcal{F}}
\def\cG{\mathcal{G}}

\def\supp{\textrm{supp}}

\def\P{{\mathsf P}} 
\def\E{{\mathsf E}} 
\def\V{{\mathsf V}} 
\def\ZZ{{\mathbb Z}}       
\def\NN{{\mathbb N}}       
\def\one{{\bf 1}}

\def\eps{\epsilon}

\def\l{\ell}
\def\la{\langle}
\def\ra{\rangle}

\def\part{\partial}

\def\d#1dt{\frac{d#1}{dt}}    



\begin{document}
\title[Strongly VRRW on a complete graph]{Strongly Vertex-Reinforced-Random-Walk on a complete graph}

\author{Michel Benaim, Olivier Raimond, Bruno Schapira}

\begin{abstract} We study Vertex-Reinforced-Random-Walk (VRRW) on a complete graph with weights of the form 
$w(n)=n^\alpha$, with $\alpha>1$. Unlike for the Edge-Reinforced-Random-Walk, which in this case localizes a.s. on $2$ sites, 
here we observe various phase transitions, and in particular localization on arbitrary large sets is possible, provided $\alpha$ is 
close enough to $1$. Our proof relies on stochastic approximation techniques. At the end of the paper, we also prove 
a general result ensuring that any strongly reinforced VRRW on any bounded degree graph localizes a.s. on a finite subgraph.  
\end{abstract}

\maketitle
\section{Introduction}

This paper considers a Vertex-Reinforced Random Walk (VRRW) on a finite  complete graph
with  weights $w_\alpha(n):= (n+1)^\alpha$ in the strongly reinforced regime $\alpha > 1.$

Such a process is a discrete time random process $(X_n)_{n\ge 0}$ living
in   $E = \{1,\dots,N\}$ and such that  for all $n \geq 0$ and $j \in E$,
$$\P(X_{n+1}=j\mid \cF_n) = \frac{w_\alpha(Z_n(j))}{\sum_{k\neq X_n} w_\alpha(Z_n(k))} \mathbf{1}_{\{X_n \neq j\}},$$
where $Z_n(j):=\sum_{\ell=0}^n 1_{\{X_\ell=j\}}$ is the number of jumps to site $j$ before time $n$,
and $\cF_n=\sigma(X_k;\;k\le n)$.

The linear regime (i.e  $\alpha=1$) has been  initially introduced by Pemantle \cite{P} on a finite graph and then extensively studied for different type of graphs by several authors (\cite{PV,T,V,LV,BT}).

\medskip
The main  result of the present paper is the following:

\bthm \label{maintheo} Let $N\ge 2$ and $\alpha>1$ be given. Then the following properties hold.
\bit
\iti
 With probability one there exists $2\le \ell \le N$,
such that $(X_n)$
visits exactly $\ell$ sites infinitely often, and the
 empirical occupation measure  converges towards the uniform measure on these $\ell$ sites.

\itii
Let $3\le k\le N.$ If $\alpha > (k-1)/(k-2)$, then the
probability to visit strictly more than $k-1$ sites infinitely often is zero.

If $\alpha < (k-1)/(k-2)$, then for any $2\le \ell \le k$, the probability
that exactly $\ell$ sites are visited infinitely often is positive.
\eit
\ethm
This result has to be compared with the situation for Edge Reinforced Random Walks (ERRW).  Limic \cite{L} (see also \cite{LT}),
proved that for any $\alpha>1$, and for  any free loops graph with bounded degree,
the ERRW with weight $w_\alpha$
visits only $2$ sites infinitely often.
It has also to be compared to the situation on the graph $\ZZ$, where for any $\alpha>1$,
the VRRW with weights $(w_\alpha(n))$ visits a.s. only $2$ sites infinitely often.

It might be interesting to notice also that when we add one loop to each site, i.e. when at each step,
independently of the actual position of the walk, the probability to jump to some site $i$ is proportional
to $w_\alpha(Z_n(i))$, then Rubin's construction (see \cite{D})
immediately shows that the walk visits a.s. only one site
infinitely often.
In fact with our techniques we can study a whole family of processes which interpolate between
these two examples: for $c\ge 0$, consider the process with transitions probabilities given by
$$\P(X_{n+1}=j\mid \cF_n) = \frac{w_\alpha(Z_n(j))\mathbf{1}_{\{X_n \neq j\}} + c\, w_\alpha(Z_n(j))\mathbf{1}_{\{X_n = j\}} }{\sum_{k\neq X_n} w_\alpha(Z_n(k)) + c\, w_\alpha(Z_n(X_n))},$$
with the same notation as above. The case $c=0$ corresponds to the VRRW on a free loop complete graph, and the case $c=1$
corresponds to the VRRW on a complete graph with loops. Then for any $c\ge 1$, the process visits a.s. only $1$ site infinitely often, and when $c\in (0,1)$, various phase transitions occur, exactly as in Theorem \ref{maintheo},
except that the critical values are this time equal to $[k-(1-c)]/[k-2(1-c)]$, for $2\le k\le N$ and localization on $1$ site is always possible and occurs
even a.s. when $\alpha> (1+c)/(2c)$. Since the proofs of these results are exactly similar as those for Theorem \ref{maintheo}, we
will not give further details here.

Finally let us observe that similar phase transitions as in Theorem \ref{maintheo} have been observed in some random graphs models,
see for instance \cite{CHJ,OS}.

The paper is organised as follows. In the next two sections we prove some general results 
for VRRW with weight $w_\alpha$, on arbitrary finite graphs, which show that in the good cases 
a.s. the empirical occupation measure of the walk converges a.s. to linearly stable equilibria of 
some ordinary differential equation. In section \ref{section-complet} we study in detail the case of 
a complete graph, we describe precisely the set of linearly stable equilibria in this case, and deduce
our main result, Theorem \ref{maintheo}. Finally in the last section we prove a general result 
ensuring that strongly reinforced VRRWs on arbitrary bounded degree graphs localize a.s. on a finite subset.

\section{A general formalism for VRRW}
We present here a general and natural framework for studying VRRW based on the formalism and results introduced in \cite{B} and \cite{BR}. Such a formalism heavily relies on stochastic approximation technics and specifically the dynamical system approach developed in \cite{B2}.


Let $A=(A_{i,j})_{i,j\le N}$ be a $N\times N$ symmetric matrix with nonnegative entries. We assume   that $A_{i,j}> 0$ for  $i\neq j$, and
 that $\sum_j A_{i,j}$ does not depend on $i$.
Let $\alpha>1$ be given. We consider the process $(X_n)_{n\ge 0}$ living in $E = \{1,\ldots,N\}$, with transition probabilities given by
\beqarr
\P(X_{n+1}=j \mid \cF_n)
= \frac{A_{X_n,j}(1+Z_n(j))^\alpha}{\sum_{k\le N} A_{X_n,k}(1+Z_n(k))^\alpha},
\eeqarr
where $Z_n$ is defined like in the introduction.
The case of the VRRW on a complete graph
is obtained by taking
\beq
\label{basicA}
A_{i,j}=1 - \delta_{ij}.
\eeq
For $i\le N$, set $v_n(i)=Z_n(i)/(n+1).$
Note  that if  $A_{i,i}=0$ and $X_k=i$ then $X_{k+1}\neq i$.
In particular for any such $i$, and $n\ge 1$, $v_n(i)\le (1/2+1/(n+1))\le 3/4$.
In other words, for all $n\ge 1$, $v_n$ belongs to the {\em reduced simplex}
 $$\Delta := \left\{v \in \mathbb{R}^N_{+}\ :\  v_i \leq 3/4 \mbox{ if } A_{ii} = 0 \mbox{ and } \sum_i v_i=1 \: \right \}.$$
In the following, we might sometimes view an element $f=(f_i)_{i\le N} \in \RR^N$ as a function on $E$,
and so we will also use the notation $f(i)$ for $f_i$.

Now for $\epsilon \in [0,1]$ and $v\in \Delta$ we let $K(\eps,v)$ denote the transition matrix defined by
\beq
\label{defKeps}
K_{i,j}(\eps,v):=\frac{A_{i,j}(\eps+v_j)^\alpha}{\sum_k A_{i,k}(\eps+v_k)^\alpha},\eeq
for all $i,j\le N.$
To shorten notation we let
\beq
\label{defK}
K(v) : = K(0,v).
\eeq
Two obvious, but key, observations are that
$$\P(X_{n+1}=j\mid \cF_n) = K_{X_n,j}((n+1)^{-1},v_n)),$$
and
$$\lim_{n \rightarrow \infty} K((n+1)^{-1},v) = K(v).$$
Hence, relying on \cite{B} and \cite{BR}, the behavior of $(v_n)$ can be analyzed through the ordinary differential equation
$\dot{v} = - v + \pi(v)$, where $\pi(v)$ is the invariant probability measure of the Markov chain with transition matrix $K(v).$


\section{The limit set theorem}
\subsection{The limiting differential equation and its equilibria}
For $v\in \Delta,$ set
$v^\alpha=(v_1^\alpha,\dots,v_N^\alpha),$
and
\beq\label{deflyap} H(v):=\sum_{i,j}A_{i,j}v_i^\alpha v_j^\alpha=\la Av^\alpha,v^\alpha\ra. \eeq
Note that $H$ is positive on $\Delta$. Hence one can define
$$\pi_i(v)=\frac{ v_i^\alpha (Av^\alpha)_i}{H(v)},  \quad i = 1, \ldots, N.$$
From the relation
$$v_j^\alpha (Av^\alpha)_i K_{i,j}(v)=A_{i,j}v_i^\alpha v_j^\alpha,$$
it follows that $K(v)$ is reversible with respect to $\pi(v).$

For $r = 0,1$ set
$T_r \Delta = \left\{v \in \mathbb{R}^N \ :\ \sum_i v_i = r\right \}$ and let
 $\imath : T_1\Delta \to \Delta$ be the map defined  by
$\imath(v)  =  \mathsf{argmin} \{\|y-v\| \ :\ y \in \Delta \}.$ Since $\Delta$ is convex
$\imath$ is a Lipschitz retraction from $T_1 \Delta$ onto $\Delta.$

Let now $ F : T_1 \Delta \to T_0 \Delta$ be the vector field defined as
\begin{equation}
\label{vecF}
 F(v)=-v+\pi(\imath(v)).
\end{equation}
Note that $F$ is Lipschitz. Thus by standard results,  $F$ induces a global flow $\Phi : \RR \times T_1\Delta \to T_1\Delta$
where for all $x \in T_1 \Delta, \, t \mapsto \Phi(t,x) := \Phi_t(x)$ is the solution to $\dot v = F(v)$ with initial condition $v(0) = x.$

For any subset $I\subset \{1,\dots,N\}$, we let
$$\Delta_I:=\{v\in \Delta\ :\ v_i=0 \quad \forall i \in E \setminus I\},$$
denote the $I$-\textit{face} of $\Delta$.
We let $int \Delta_I = \{v \in \Delta_I \: : v_i > 0  \quad \forall i \in  I\}$ denote the (relative) interior of $\Delta_I.$
For $v \in \Delta$ we let $supp(v) = \{i \in E \: : v_i \neq 0\},$ so that $v$ always lies in the $supp(v)$-\textit{face}
of $\Delta.$
\blem
The flow $\Phi$ leaves $\Delta$ positively invariant: $\forall t \geq 0, \, \Phi_t(\Delta) \subset \Delta$; and

for each $I \subset E$, the face $\Delta_I$ is locally invariant: $\forall v \in \Delta_I, \,  \forall t \in \RR$,
$\Phi_t(v) \in \Delta \Leftrightarrow\Phi_t(v) \in \Delta_I$.
\elem

\begin{proof}
For all $v \in \Delta$, $\pi(v)$ lies in $\Delta$. Indeed for the Markov chain having transition matrix  $K(v)$ the empirical occupation measure lies in $\Delta$ and by the ergodic theorem (for finite Markov chains) the same is true for $\pi(v).$
Hence $F(v)$ points inward $\Delta$ for all $v \in \Delta$, proving that $\Delta$ is positively invariant.
Since $F_i(v) = 0$ when $v_i = 0$, each face is locally invariant.
\end{proof}

Let $$\mathcal{C} = \{v \in \Delta \: : F(v) = 0 \},$$
denote the {\em equilibria set} of $F$.
Relying on stochastic approximation theory \cite{B,B2,BR} it will be shown below (Proposition \ref{converge}) that $(v_n,n\ge 1)$
converges almost surely to $\mathcal{C}$.

The next result is similar to the case $\alpha=1$ (see for instance \cite{P}):
\blem
\label{lemeq}
The map  $H : \Delta \to \RR $ is a {\em strict Lyapunov function}, meaning that
$\la \nabla H(v),F(v)\ra$ is positive for all $v\in \Delta \setminus \mathcal{C}$.
\elem
\begin{proof}  One has
$\partial_i H(v)=2\alpha  v_i^{\alpha-1} (Av^\alpha)_i$. Thus
\beqarr
\la \nabla H(v),F(v)\ra
&=& \sum_i 2\alpha  v_i^{\alpha-1} (Av^\alpha)_i \left(-v_i + \frac{ v_i^\alpha (Av^\alpha)_i}{H(v)}\right)\\
&=& \frac{2\alpha}{H(v)} \left( - \big{(}\sum_i  v_i^{\alpha-1} (Av^\alpha)_i  v_i\big{)} ^2 +  \sum_i \big{(}v_i^{\alpha-1} (Av^\alpha)_i\big{)}^2 v_i \right)\\
&\ge& 0,
\eeqarr
with equality only when $v_i^{\alpha-1} (Av^\alpha)_i$ does not depend on $i\in \supp(v)$, i.e. only when $v$ is an equilibrium.
\end{proof}
\brem {\em The barycenter $v=(1/N,\dots,1/N)$ is always an equilibrium. }\erem
\blem $H(\mathcal{C})$ has empty interior. \elem
\begin{proof} The computation of $\partial_i H(v)$ shows that
$$(\partial_i - \partial_j)H(v) =  2\alpha  (v_i^{\alpha-1} (Av^\alpha)_i -  v_j^{\alpha-1} (Av^\alpha)_j).$$
Hence, for all $v$ in the relative interior of $\Delta,$ $F(v) = 0 \Leftrightarrow \nabla H(v) = 0.$
In other word $\mathcal{C} \cap int{\Delta} = \nabla H^{-1}(0)  \cap int{\Delta}.$
By Sard's theorem, it follows that $H(\mathcal{C} \cap int{\Delta})$ has measure zero, hence empty interior.
Similarly, for each face $I$, $H(\mathcal{C} \cap int{\Delta_I})$ has empty interior. This proves the lemma.
\end{proof}
 \bprop
 \label{converge}
 The set of limit points of $(v_n)$ is a connected subset of $\mathcal{C}$.
 \eprop
\begin{proof} By proposition 3.3 and Theorem 3.4 in \cite{B} or Proposition 4.6 in \cite{BR}, we get that the limit set of $(v(n))$ is an internally chain transitive
set for $\Phi$.
Since $H$ is a strict Lyapunov function and $H(\mathcal{C})$ has empty interior,
it follows from Proposition 6.4 in \cite{B2} that such a limit set is contained in  $\mathcal{C}$.
\end{proof}

In particular, when all the equilibria of $F$ are isolated, then $v_n$ converges a.s. toward one of them, as $n\to \infty$.

\brem
{\em When $A$ is not symmetric, the convergence result given by Proposition \ref{converge} fails to hold.
Indeed, an example is constructed   in \cite{B}  with $N = 3$ and  $\alpha = 1$
for which the limit set of $(v_n)_{n\ge 1}$ equals $\partial \Delta$.
This behavior persists for  $\alpha = 1 + \epsilon$ and $\epsilon > 0$ small enough. }
\erem

\subsection{Stable and unstable equilibria.}
An   equilibrium $v$ is called {\it linearly stable} provided all the eigenvalues of $DF(v)$, the differential of $F$ at
point $v$, have negative real parts.
It is called {\it linearly unstable} if one of its eigenvalues has a positive real part.

Now we will see with the next result that to study the stability or instability
of the equilibria, it suffices in fact to consider only those which belong to the interior of $\Delta$.
In the following we let $(e_1,\dots,e_N)$ denote the canonical basis of $\RR^N$.

\blem \label{stabface} Let $v$ be an equilibrium. Then, for $i,j\in \supp(v)$, we have
$$D_{e_i-e_j}F(v)=(\alpha - 1) (e_i-e_j) +\alpha\sum_{\ell} \frac{v_{\ell}^\alpha (A_{\ell,i} v_i^{\alpha-1}-A_{\l,j} v_j^{\alpha-1})}{H(v)} e_{\ell},$$
and for $i\notin \supp(v)$,
$$D_{e_i-v}F(v)=-(e_i-v).$$
Furthermore, the eigenvalues of $DF(v)$ are all reals.
\elem
\begin{proof}
For any $i,j\le N$, and $v\in \Delta$,
\beqarr
\partial_j (Av^\alpha)_i &=& \alpha A_{i,j}v_j^{\alpha -1},
\eeqarr
and then by using that $A$ is symmetric, we get
\beqarr
\partial_j H(v) &=& 2\alpha v_j^{\alpha-1} (Av^\alpha)_j.
\eeqarr
Thus
\beqarr
\partial_j \pi_i(v) &=& \delta_{i,j} \alpha \frac{ v_i^{\alpha-1} (Av^\alpha)_i}{H(v)} + \alpha \frac{ v_i^{\alpha} A_{i,j}v^{\alpha-1}_j}{H(v)}\\
&& -2\alpha \frac{ v_i^{\alpha} (Av^\alpha)_i}{H^2(v)} v_j^{\alpha-1} (Av^\alpha)_j.
\eeqarr
Now assume that $v$ is an equilibrium, and let $i,j\in \supp(v)$. We get with Lemma \ref{lemeq}
\beqarr
\partial_j \pi_i(v) &=& \delta_{i,j} \alpha+ \alpha \frac{ v_i^{\alpha} A_{i,j}v^{\alpha-1}_j}{H(v)} -2\alpha v_i,
\eeqarr
and then
$$\partial_j F_i(v) = \delta_{i,j} (\alpha-1)+ \alpha \frac{ v_i^{\alpha} A_{i,j}v^{\alpha-1}_j}{H(v)}-2\alpha v_i.$$
On the other hand if $v_i=0$ or $v_j=0$, then
$$\partial_j \pi_i(v)=0,$$
and thus
$$\partial_j F_i(v) = - \delta_{i,j}.$$
The first part of the lemma follows. To see that eigenvalues are real, we may assume without loss of generality that $v \in int(\Delta).$ Note that $\partial_j F_i(v) v_j = \partial_i F_j(v) v_i.$
Therefore, the transpose of $DF(v)$ is self adjoint with respect to the dot product $(x,y) = \sum_i v_i x_i y_i$, and this concludes the proof of the lemma.
\end{proof}

\medskip
\noindent As announced above we deduce
\bcor
\label{corstabface}
An equilibrium on a face is linearly stable (respectively unstable), if and only if, it is so
for the restriction of $F$ to this face.
\ecor
\begin{proof}
Indeed, assume that $v$ is an equilibrium on a face $\Delta_I$ associated to some subset $I$. Then
the previous lemma shows that for any $i\notin I$, $e_i-v$ is a stable direction. So the result
of the corollary follows from our definitions of stable and unstable equilibria.
\end{proof}

In other words to study the stability or instability of equilibria, and we will see in the next two subsections why this
question is important, it suffices to consider those belonging to the
interior of $\Delta$.

\subsection{Non convergence towards unstable equilibria}
The purpose of this section is to prove the following result.
\bthm \label{nonconv}
Let $v^*$ be a linearly unstable equilibrium.
Then the probability that $v_n$ converges towards $v^*$ is equal to $0$.
\ethm
\begin{proof}
Let us recall now that for $g\in \RR^N$ and $i\in E$, we use the notation $g(i)=g_i$.
For $u,v \in \RR^N$, we also set $uv := \sum_i u_iv_i$, and $\|u\|=\sup_i |u_i|$.
Furthermore, $C$ will denote a non-random constant that may vary from lines to lines.

For $v\in \Delta$, let $Q(v)$ be the pseudo-inverse of $K(v)$ defined by:
$$(I-K(v)) Q(v)g = Q(v) (I-K(v))g = g - (\pi(v) g) \mathbf{1},$$
for all $g\in \RR^N$, with $I$ is the identity matrix and $\mathbf{1}(i)=1$ for all $i\in E$.
Then by a direct application of the implicit function theorems (see Lemma 5.1 in \cite{B}) one has
\blem \label{propQv}
For any $v\in \Delta$ and $i\in E$,
$Q(v)$, $K(v)Q(v)$, and $(\partial/\partial_{v_i}) (K(v)Q(v))$, are bounded operators on $\ell^\infty(E)$,
the space of bounded functions on $E$, and their norms are uniformly bounded in $v \in \Delta$.
\elem

Now for all $n\ge 1$ and $i\in E$, we can write
$$v_{n+1}(i)-v_n(i) = \frac{1}{n+1}\big(-v_n(i) + e_i(X_{n+1})\big).$$
Note that $K\big((n+1)^{-1},v_n\big) = K(\tilde{v}_n)$, where
$$\tilde v_n(i) = \frac{v_n(i) + 1/(n+1)}{1+N/(n+1)} \quad \hbox{ for } 1\le i\le N.$$
Note also that $\|\tilde{v}_n-v_n\|\le C/n$. Let next $z_n$ be defined by
$$z_n(i)=v_n(i) + \frac{K(\tilde{v}_n)Q(\tilde{v}_n)e_i(X_n)}{n} \quad \hbox{ for } 1\le i\le N.$$
Then Lemma \ref{propQv} implies that $\|z_n-v_n\|\le C/n$. Moreover, we can write:
\beq
\label{robbins}
z_{n+1}-z_n = \frac{F(z_n)}{n+1} + \frac{\eps_{n+1}}{n+1} + \frac {r_{n+1}}{n+1},
\eeq
where $\eps_{n+1}$ and $r_{n+1}$ are such that for all $1\le i\le N$,
$$\eps_{n+1}(i):=Q(\tilde{v}_n)e_i(X_{n+1})-K(\tilde{v}_n)Q(\tilde{v}_n)e_i(X_n),$$
and $r_{n+1} = \sum_{k=1}^4 r_{n+1,k}$, with
\beqarr
r_{n+1,1} &=& F(v_n)-F(z_n) \\
r_{n+1,2} &=& \pi(\tilde{v}_n) - \pi(v_n) \\
r_{n+1,3}(i) &=& K(\tilde{v}_n)Q(\tilde{v}_n)e_i(X_{n})\, \left(1 - \frac {n+1} n \right)\\
r_{n+1,4}(i) &=& K(\tilde{v}_{n+1})Q(\tilde{v}_{n+1})e_i(X_{n+1})-K(\tilde{v}_n)Q(\tilde{v}_n)e_i(X_{n+1}),
\eeqarr
for $1\le i\le N$.

By using the facts that $F$ and $\pi$ are Lipschitz functions on $\Delta$,  $\|v_n-\tilde{v}_n\|+\|v_n-z_n\| \le C/n$, and by applying Lemma \ref{propQv},
we deduce that
$$\|r_{n+1}\| \le C/n.$$
Moreover, we have
$$\E[\eps_{n+1}\mid \cF_n]=0.$$
Since $v^*$ is linearly unstable there exists, by Lemma \ref{stabface}, $f \in T_0 \Delta$ and $\lambda > 0$ such that
 \bdes
\iti $D_f F (v) = \lambda f$ and,
\itii  $ f_i = 0$ for $i \in E \setminus supp(v^*).$
 \edes

Such an $f$ being fixed, we claim that on the event $\{v_n\to v^*\}$,
\begin{eqnarray}
\label{liminf}
\liminf_{n\to \infty}\, \left\{\E[(\eps_{n+2}f)^2\mid \cF_{n+1}]+ \EE[(\eps_{n+1}f)^2\mid \cF_n]\right\}> 0.
\end{eqnarray}
To prove this claim, we first introduce some  notation:
for $\mu$ a probability measure on $E$, and $g\in\RR^N$, denote by $\V_\mu(g)$ the variance of $g$
with respect to $\mu$
$$\V_\mu(g) :=   \frac 12 \sum_{1\le j,k\le N} \mu(j)\mu(k) \big(g(j)-g(k)\big)^2.$$
Then for any $n\ge 0$ and $i\le N$, let $\mu_{n,i}$ be
the probability measure defined by  $\mu_{n,i}(j) = K_{i,j}(\tilde{v}_n)$. Set also $V_n(i):=\V_{\mu_{n,i}}(Q(\tilde v_n)f)$.
Then we have that
$$\E[(\eps_{n+1}f)^2\mid \cF_n] = V_n(X_n).$$
Furthermore, when $v_n$ converges toward $v^*$, $K(\tilde{v}_n)$ and $Q(\tilde{v}_n)$ converge respectively toward $K(v^*)$ and $Q(v^*)$.
Thus, for any $i\in E$,
$$ \liminf_{n\to \infty}\, V_n(i) \ge  V^*(i):= \V_{\mu^*_i}(Q(v^*)f),$$
where $\mu^*_i(j)=K_{i,j}(v^*)$, for all $j\le N$.
Next by using the fact that when $X_n=i$ and $A_{i,i}=0$, then $X_{n+1}\neq X_n$,
we get that
$$\liminf \left\{\E[(\eps_{n+2}f)^2\mid \cF_{n+1}]+\E[(\eps_{n+1}f)^2\mid \cF_n]\right\}\ge \min_i\, c^*(i),$$
where
\begin{eqnarray*}
c^*(i) &:=& \min_{j \in \cA_i } (V^*(i)+V^*(j)),
\end{eqnarray*}
and
\begin{eqnarray*}
\cA_i : = \left\{ \begin{array}{ll}
E & \textrm{if }A_{i,i}\neq 0 \\
E\smallsetminus\{i\} & \textrm{if }A_{i,i} = 0.
\end{array}
\right.
\end{eqnarray*}
Now by using that
$$Q(v^*)f - K(v^*)Q(v^*)f = f- (v^*f)\mathbf{1}$$
(recall that $\pi(v^*)=v^*$),
we see that $Q(v^*)f$ has constant coordinates on $\supp(v^*)$,
if and only if, $f$ has constant coordinates on $\supp(v^*)$.
But since $f\in T_0\Delta$ and $f_i = 0$ for $i \not \in \supp(v^*)$; this cannot be the case.
Since $\mu^*_i(j)>0$, when $j\neq i$ and $j\in \supp(v^*)$, it follows already that $c^*(i) >0$,
for all $i\notin \supp(v^*)$. Now let $i\in \supp(v^*)$ be given. If $A_{i,i}\neq 0$, then again we have $V^*(i)>0$,
and thus $c^*(i)>0$. Now assume that $A_{i,i}=0$. If $\# \supp(v^*)\ge 3$, then there can
be at most one value of $i$, for which
$V^*(i)=0$, and thus in this case we have $c^*(i)>0$ as well.
Let us then consider the case when
$\# \supp(v^*) = 2$, and say $\supp(v^*)= \{i,j\}$. Recall that if $A_{i,i}\neq 0$, then $V^*(i)>0$.
However, we cannot have $A_{i,i}=A_{j,j}=0$, since otherwise $v^*_i=v^*_j=1/2$ and 
by lemma \ref{stabface} $v^*$ is linearly stable. 
Finally we have proved that in any case $\min_i\, c^*(i)>0$.
Theorem \ref{nonconv} is then a consequence of \eqref{liminf} and Corollary 3.IV.15 p.126 in \cite{Du}.
\end{proof}

\subsection{Convergence towards stable equilibria and localization}

\bthm 
\label{convpos}
Let $v^*$ be a linearly stable equilibrium. Then the probability that $v_n$ converges towards $v^*$ is positive.
\ethm
\begin{proof} follows from Corollary 6.5 in \cite{B} since any linearly stable equilibrium is a minimal attractor. \end{proof}

\bthm 
\label{convpos2}
Let $v^*\in \Delta$ be a linearly stable equilibrium.
Then a.s. on the event $\{\lim_{n\to \infty} v_n = v^*\}$, the set $E \setminus \supp(v^*)$ is visited only finitely many times.
\ethm
The proof follows directly from the next two lemmas:
\blem There exists $\nu>0$ such that on the event $\{\lim_{n\to \infty} v_n = v^*\}$,
$$\lim_{n\to\infty} n^\nu \|v_n-v^*\| = 0.$$
\elem
\begin{proof} This is similar to Lemma 8 in \cite{BT}. We give here an alternative and more direct proof relying on \cite{B2}.
Since $v^*$ is a linearly stable  equilibrium there exists a neighborhood $B$ of $v^*,$ a constant $C > 0$ and $\lambda > 0$,
such that $\|\Phi_t(v)-v^*\| \leq C e^{-\lambda t}$, for all $v \in B$ (see e.g \cite{R}, Theorem 5.1).
Let $\tau_n = \sum_{k = 1}^n 1/k$ and let  $V : \RR^+ \to \Delta$ denote the continuous time piecewise affine process defined by
a) $V(\tau_n) = z_n$ and  b) $V$ is affine on $[\tau_n, \tau_{n+1}]$.
By  (\ref{robbins}) and Doob's inequalities, the interpolated process $V$  is almost surely a {\em $-1/2$ asymptotic pseudo trajectory} of $\Phi$, meaning that
$$\limsup_{t \rightarrow \infty} \frac{1}{t} \log \left(\sup_{0 \leq h \leq T} \|\Phi_h(V(t) - V(t+h)\|\right) \leq -1/2$$ for all $T > 0.$
For a proof of this later assertion see \cite{B2}, Proposition 8.3.
Now, by Lemma 8.7 in \cite{B2}
$$\limsup_{t \rightarrow \infty} \frac{1}{t}\log ( \|V(t)-v^*\| ) \leq  - \min (1/2,\lambda)$$
on the event $\{v_n \rightarrow v^*\}$.
This proves that $\|z_n-v^*\|=O(n^{- \min(1/2,\lambda}))$,
which concludes the proof of the lemma. \end{proof}

\begin{lemma}
For any $I\subseteq \{1,\dots,N\}$, and $\nu\in (0,1)$, a.s. on the event
$$E_\nu(I) := \{ \lim_{n\to \infty}\ v_n(i)\, n^\nu =0  \quad \forall i\in I\},$$
the set $I$ is visited only finitely many times.
\end{lemma}
\begin{proof}
For $m\ge 1$, set
$$E_{m,\nu}(I):=\{|v_k(i)| \le k^{-\nu} \quad \forall k\ge m\quad \forall i\in I\}.$$
Note that on $E_{m,\nu}(I)$, at each time $k\ge m$, the probability to jump to some vertex $i\in I$, is bounded above
by $p_k:=N^{1+\alpha}\, k^{-\alpha\nu}$. Let now $(\xi_k)_{k\ge m}$ denotes some sequence of independent Bernoulli random variables with
respective parameters $(p_k)_{k\ge m}$. Then for any $n\ge m$, on $E_{m,\nu}(I)$ the number of jumps on $I$ between time
$m$ and $n$ is stochastically dominated by
$$Z_n:=\sum_{k=m}^n \, \xi_k.$$
However, it is well known that a.s. $\limsup Z_n/n^{1-\nu'}<\infty$, for any $\nu'< \alpha \nu\wedge 1$.
We deduce that a.s. for any $\nu'<\alpha\nu \wedge 1$,
$$E_{m,\nu}(I) \ \subseteq \ E_{\nu'}(I).$$
Since $E_\nu(I) \subseteq \cup_m E_{m,\nu(I)}$, we deduce that a.s. for any $\nu'< \alpha\nu \wedge 1$,
$$E_\nu(I) \ \subseteq \ E_{\nu'}(I).$$
Since $\alpha>1$, it follows by induction that a.s.
$$E_\nu(I) \ \subseteq \ E_{\beta}(I),$$
for any $\beta \in (1/\alpha,1)$.
But a simple application of the Borel-Cantelli lemma shows that for any such $\beta$,
a.s. on $E_{\beta}(I)$, the set $I$ is visited
only finitely many times. This concludes the proof of the lemma.
\end{proof}

\section{The case of the VRRW on a complete graph}
\label{section-complet}
In this section we study in detail the case of the VRRW on a complete graph described in
the introduction. In other words,  $A$ is given by (\ref{basicA}).

Since the case $N=2$ is trivial, we assume in all this section that $N\ge 3$.

\medskip
We first study the stability of the centers of the faces.
As  already explained, this reduces to analyze  the center of $\Delta$.

\blem Let $v=(1/N,\dots,1/N)$ be the center of $\Delta$.
Then $v$ is a linearly stable (respectively unstable) equilibrium  if $\alpha < (N-1)/(N-2)$,
(respectively $\alpha > (N-1)/(N-2)$).
\elem
\begin{proof}
Lemma \ref{stabface} shows that for all $i\neq j$,
$$D_{e_i-e_j}F(v) = \left(-1+\alpha\left(\frac{N-2}{N-1}\right)\right)(e_i-e_j).$$
The lemma follows immediately.  \end{proof}

\medskip
\noindent By combining this lemma with Corollary \ref{corstabface}, we get
\blem
\label{centerface}
Let $v$ be the center of the face $\Delta_I$ associated with some subset $I$ with cardinality $k\le N$.
Then $v$ is a linearly stable (respectively unstable) equilibrium  if $\alpha< (k-1)/(k-2)$, (respectively $\alpha>(k-1)/(k-2)$).
\elem

It remains to study the stability of the other equilibria. We will see that they are all unstable, which will conclude
the proof of Theorem \ref{maintheo}.

First we need the following lemma, which shows that coordinates of equilibriums take at most two different values.

\blem \label{lem43}
Let $v$ be an equilibrium in the interior of $\Delta$, which is different from its center.
Then $\# \{v_i\ :\ i\le N\}=2$.
\elem
\begin{proof} Let $a=\sum_i v_i^\alpha$ and $b=\sum_i v_i^{2\alpha}$. Since $v$ is an equilibrium, we have for all $i$,
$$v_i=v_i^{\alpha}\, \frac{a-v_i^\alpha}{a^2-b}.$$
Since all coordinates of $v$ are positive by hypothesis, this is equivalent to
$$f(v_i)=a-b/a,$$
for all $i$, where $f(x)=-x^{2\alpha-1}/a+x^{\alpha-1}$.
Now observe that
$$f'(x)=x^{\alpha-2}\big(-(2\alpha -1)x^{\alpha}/a+(\alpha-1)\big),$$
does not vanish on $(0,1)$ if $a\ge(2\alpha-1)/(\alpha-1)$,
and vanishes in exactly one point otherwise. Thus for any fixed $\lambda\in \RR$,
the equation $f(x)=\lambda$,
has at most one solution in $(0,1)$, if $a\ge(2\alpha-1)/(\alpha-1)$, and at most two otherwise.
The lemma follows.
\end{proof}

\medskip
Let $v$ be an equilibrium in the interior of $\Delta$, which is different from its center.
Lemma \ref{lem43} shows that its coordinates take exactly two different values, say $u_1$ and $u_2$.
Since the action of permutation of the coordinates commutes with $F$,
we can always assume w.l.g. that $v_i=u_1$, for $i\le k$, and $v_i=u_2$, for $i>k$,
for some $k\le N/2$.
Denote now by $E_k$ the set of such equilibria (those in the interior of $\Delta$, not equal to the center of $\Delta$,
and having their first $k$ coordinates equal as well as their last $N-k$ coordinates).
For $v\in E_k$, we also set $t(v)=v_N/v_1$. We have the

\blem \label{unstable1} Assume that $\alpha\ge (N-1)/(N-2)$, and let $v \in E_k$, with $k\le N/2$.
If $k>1$ or if $t(v)<1$, then $v$ is linearly unstable.
\elem
\begin{proof}
It follows from Lemma \ref{stabface} that for any $i<j\le k$,
$$D_{e_j-e_i}F(v) = \lambda_1 (e_j-e_i),$$
and for $k+1\le i<j\le N $,
$$D_{e_j-e_i}F(v) = \lambda_2 (e_j-e_i),$$
with
$$\lambda_m=\left( (\alpha-1)-\alpha \frac{u_m^{2\alpha-1}}{H(v)}\right),$$
for $m=1,2$. Then it just suffice to observe that if $u_1<u_2$,
$\lambda_1>0$, whereas if $u_1>u_2$, $\lambda_2>0$. Thus if any of the hypotheses of the lemma is satisfied,
there $DF(v)$ has at least one positive eigenvalue, and so $v$ is linearly unstable.
\end{proof}
\noindent On the other hand we have the following:
\blem \label{lem46} Let $v \in E_k$ be given. If $k=1$ and $\alpha\ge (N-1)/(N-2)$, then $t(v)<1$ and $v$ is unstable.
Similarly, if $\alpha < (N-1)/(N-2)$, then $v$ is linearly unstable.
\elem
\begin{proof}
Following our previous notation, set $u_1:=v_1$ and $u_2:=v_N$, and recall that $u_1\neq u_2$, since $v$ is not equal
to the center of $\Delta$. Recall also that $t(v)=u_2/u_1$. Note that since $ku_1+(N-k)u_2=1$, we have
$u_1=1/(k+(N-k)t)$. Now the fact that $v$ is an equilibrium
means that $F(v)=0$, which is equivalent to say that the function $\varphi$ defined by
$$\varphi(t):=-(N-k-1)t^{2\alpha-1} +(N-k)t^\alpha -k t^{\alpha-1} + (k-1),$$
vanishes at point $t(v)$.
We now study the function $\varphi$. First $\varphi'(t)=t^{\alpha-2} \psi(t)$, with
$$\psi(t)=-(2\alpha-1)(N-k-1)t^\alpha +\alpha (N-k)t-(\alpha-1)k.$$
In particular $\psi$ is strictly concave, $\psi(0)<0$ and $\lim_{t\to \infty} \psi(t) = -\infty$.
Now two cases may appear. Either $\psi$ vanishes in at most one point, in which case $\varphi$ is nonincreasing, thus vanishes
only in $1$. But this case is excluded, since $t(v)\neq 1$ by hypothesis. In the other case
$\psi$ vanishes in exactly two points, which means that there exist $t_1<t_2$ such that $\varphi$ is decreasing in
$(0,t_1)\cup(t_2,\infty)$, and increasing in $(t_1,t_2)$.
Now consider first the case when $k=1$, which implies $\varphi(0)=0$.
If $\alpha\ge (N-1)/(N-2)$, then $\varphi'(1)\le 0$, thus $\varphi$ has at most one zero in $(0,1)$ and no zero
in $(0,\infty)$. Together with Lemma \ref{unstable1} this proves the first part of the lemma. If $\alpha>(N-1)/(N-2)$,
then $\varphi'(1)>0$, and thus $\varphi$ has no zero in $(0,1)$ and exactly one zero in $(0,\infty)$, in which the derivative
of $\varphi$ is negative. The fact that this zero corresponds to an unstable equilibrium will then follow from \eqref{fact} below.
But before we prove this fact, let us consider now the case $k>1$ and $\alpha<(N-1)/(N-2)$, which
imply that $\varphi(0)>0$ and $\varphi'(1)>0$. Thus $\varphi$ vanishes in exactly one point in $(0,1)$
and another one in $(1,\infty)$, and again the derivative of $\varphi$ in these points is negative.
So all that remains to do is to prove the following fact
\begin{eqnarray}
 \label{fact}
\textrm{If $v\in E_k$ is such that $\varphi'(t(v))<0$, then $v$ is unstable.}
\end{eqnarray}
Let us prove it now. For $t\in(0,\infty)$, set $u_1(t)=1/(k+(N-k)t)$ and $u_2(t) = tu_1(t)$.
Then let $v(t)\in (0,1)^N$, be the point whose first $k$ coordinates are equal to $u_1(t)$ and whose last $N-k$
coordinates are equal to $u_2(t)$. Then we have
$$H(v(t)) = k(k-1)u_1(t)^{2\alpha}+2k(N-k) t^\alpha u_1(t)^{2\alpha} + (N-k)(N-k-1)t^{2\alpha}u_1(t)^{2\alpha},$$
and after some computation we find that
$$F(v(t))= -\frac{tu_1(t)^{2\alpha+1}}{H(v(t))}\, \varphi(t)\ e_k,$$
where $e_k$ is the vector, whose first $k$ coordinates are equal to $-(N-k)$ and whose last $N-k$
coordinates are equal to $k$. Then notice that $D_{e_k}F(v(t))=c(d/dt)F(v(t))$, for some constant $c>0$.
But recall that when
$v(t)$ is an equilibrium, $\varphi(t) = 0$. Thus
$$\frac d{dt} F(v(t)) = - \frac{tu_1(t)^{2\alpha+1}}{H(v(t))}\, \varphi'(t)\, e_k.$$
So if $\varphi'(t)<0$, $e_k$ is an unstable direction, proving \eqref{fact}. This concludes the proof of the lemma.
\end{proof}

\medskip
\blem \label{lem44} The number of equilibriums in $\Delta$ is finite. \elem
\begin{proof}
Let $v\in \Delta$, such that that $v_i=u_1>0$, for $i\le k$, and $v_i=u_2>0$, for $i>k$,
for some $k\le N/2$. Set $t=u_2/u_1$, so that $u_1=(k+\ell t)^{-1}$ and $u_2=t u_1$ and where $\ell =N-k$. 
Recall (see the proof of Lemma \ref{lem46}) that $v$ is an equilibrium if and only if $\phi(t)=0$, where $\phi:]0,\infty[\to \RR$ is defined by
$$\phi(t)=(\ell-1) t^{2\alpha-1} - \ell t^\alpha + k t^{\alpha -1} -(k-1).$$
Studying $\phi$, we show that two situations may occur: $1$ is the only solution or there exist two other solutions, distinct from $1$. This proves that $E_k$ contains at most $2$ equilibriums. 
And therefore, one can check that the number of equilibriums in the interior of $\Delta$ is at most $2^N-1$. 
\end{proof}

\noindent Lemmas \ref{unstable1} and \ref{lem46} show that any equilibrium in the interior of $\Delta$, which is not equal to the center of $\Delta$ is linearly unstable. Together with Lemma \ref{centerface}, Lemma \ref{lem44} and Theorems \ref{nonconv}, \ref{convpos} and \ref{convpos2}  this concludes the proof of Theorem \ref{maintheo}. $\square$

\section{Almost sure localization}
In this section we prove a general result on the localization of 
strongly reinforced VRRWs. First we recall that given a locally finite graph $G$, and $w:\NN\to (0,\infty)$,  
we can define the VRRW $(X_n)_{n\ge 0}$ by 
$$\P(X_{n+1}=y\mid \cF_n) = \frac{w(Z_n(y))}{\sum_{z\sim X_n} w(Z_n(z))} \mathbf{1}_{\{X_n\sim y\}},$$
with $Z_n(\cdot)$ as in the introduction, and where the symbol $X_n\sim y$ means that $X_n$ and $y$ are neighbours in $G$. 
We say that the VRRW is strongly reinforced when 
\begin{eqnarray}
 \label{SRweight} 
\sum_{n=0}^\infty \frac 1{w(n)} <\infty. 
\end{eqnarray}
It is easy to see that under this hypothesis the walk can localize on two sites with positive probability. 
But something more can be said: 

\begin{theorem}
Let $w:\NN\to (0,\infty)$ be a weight function satisfying \eqref{SRweight}, 
and let $G$ be a graph with bounded degree. Then a.s. the VRRW on $G$ with weight $w$ visits only a finite number of sites. 
\end{theorem}

\begin{proof}
We adapt an argument of Volkov who proved the result when the graph is a tree (see \cite{V1}). 
Let $d>0$ be such that any vertex in $G$ has degree at most $d$ and let $\varepsilon \in (0,1)$ be such that 
$\varepsilon d^3 <1$. 

\vspace{0.2cm}
It will be convenient to introduce a Rubin's construction of the walk. Similar ones were introduced already by Sellke \cite{Sel} and Tarr\`es \cite{T2}, see also \cite{BSS}. 

\vspace{0.2cm}
To each pair of neighbours $x\sim y$ of $G$ we associate two oriented edges $(x,y)$ and $(y,x)$, 
and denote by $\cE$ the set of oriented edges obtained in this way. 
Then we consider a sequence 
$(\xi^e_\ell)_{e\in \cE, \ell \ge 0}$ of independent exponential random variables, such that for any $e$ and $\ell$, 
the mean of $\xi_\ell^e$ equals $1/w(\ell)$. We will also use the notation: 
$$\xi_x=(\xi_\ell^{(x,y)})_{y\sim x, \ell \ge 0},$$
for any vertex $x$ of $G$. Then given the sequence $(\xi_x)_{x\in G}$, and a vertex $x_0$ of $G$, 
we construct a continuous time process $(\tilde X_t)_{t\ge 0}$
as follows. 

\vspace{0.2cm}
First $\tilde X_0=x_0$. Then for each oriented edge $(x_0,y)$, we let a clock ring at time $\xi_0^{(x_0,y)}$, and 
let the process jump along the edge corresponding to the clock which rings first (at time 
$\min_{y\sim x_0} \xi_0^{(x_0,y)}$). At this time we also stop all clocks associated to the edges going out of $x$.  

\vspace{0.2cm}
Now assuming that $(\tilde X_t)$ has been constructed up to some jumping time $t$, and that $\tilde X_t=x$, 
we use the following rule. 
For any $y\sim x$, if during the time between the last visit to $x$ before $t$
(which we take to be $0$ if $t$ is the first time the walk visits $x$) and $t$ the walk has visited $y$, then start 
a new clock associated to the edge $(x,y)$ which will ring at time $\xi_k^{(x,y)}$, 
with $k$ the number of visits to $y$ before time $t$. If not, then restart the clock associated to $(x,y)$ which has been stopped 
after the last visit of the walk in $x$, or start a new clock which will ring after a time $\xi_0^{(x,y)}$ if both $x$ and $y$ 
have never been visited yet before time $t$. Then the walk jumps again along the edge associated to the clock which rings first at 
the time when it rings, and then we stop all clocks associated to edges going out of $x$. Then it is a basic and well known fact that a.s. two clocks never ring at the same time, so that this construction is a.s. well defined, and that if $(\tau_k)_{k\ge 0}$ is the sequence of jumping times (with $\tau_0=0$), then $(\tilde X_{\tau_k})_{k\ge 0}$ has the same law  
as the process $X$ starting from $x_0$. 

\vspace{0.2cm}
Now for any vertex $x$ and integer $N\ge 1$ we consider the following event 
$$E_x(N):= \left\{ \sum_{\ell \ge N} \xi_\ell^{(x,y)} < \min_{z\sim x} \xi_0^{(x,z)}\quad \forall y\sim x\right\}.$$
Note that, denoting by $d_x$ the degree of $x$,
\begin{eqnarray*}
\P(E_x(N)) &=& \prod_{y \sim x}\E\left[ e^{-d_x w(0) \sum_{\ell \ge N} \xi_\ell^{(x,y)}}\right] \\
&=&  \prod_{l \geq N}\left( 1-\frac{d_x w(0)}{w(\ell)}\right)^{d_x} \ge \prod_{l \geq N}\left( 1-\frac{d w(0)}{w(\ell)}\right)^d
\end{eqnarray*}
using in the last inequality that $d_x\le d$.
So that, when \eqref{SRweight} is satisfied,  we can fix $N$ such that $\P(E_x(N)) \ge 1-\varepsilon$, for all $x$.

\vspace{0.2cm}
To each vertex $x$, let $\tau_x$ be the first time the walk visits $x$. 
When $\tau_x<\infty$, we associate to $x$ a type:
\begin{itemize}
\item the type $1$ if at time $\tau_x$ one of its neighbours has been visited more than $N$ times;
\item the type $2$ if $x$ is not of type $1$ and if every neighbour of $x$ has at least one of its neighbours which has been visited more than $N$ times before $\tau_x$;
\item the type $3$ otherwise.
\end{itemize}
We also associate two sets of vertices, $C^1_x$ and $C^2_x$\,: 
when $x$ is of type $1$  and on $E_x(N)$,  $C^1_x=\emptyset$,   
otherwise $C^1_x$ is the set of all neighbours of $x$ that have not been visited by the walk at time $\tau_x$.
Thus $C^1_x$ contains all the vertices that may be visited for the first time after a jump from $x$.

Define next $C^2_x:=\cup_{y\in C^1_x} C^2_{x,y}$, with $C^2_{x,y}$ as follows. When $y$ has a neighbour that has been visited more than $N$ times at time $\tau_x$ and on $E_y(N)$, $C^2_{x,y}=\emptyset$,   
otherwise $C^2_{x,y}$  is the set of all neighbours of $y$ that have not been visited by the walk at time $\tau_x$.
Thus $C^2_x$ contains all the vertices that may be visited for the first time after a jump from a vertex in $C^1_x$.

\vspace{0.2cm}
Observe that when $x$ is of type $3$,
conditionally to $\cF_{\tau_x}$, the walk has a probability bounded below by some constant $p>0$, depending only on $N$, in particular independent of $x$ and $\cF_{\tau_x}$, to be stuck on $\{x,y\}$ for some $y\sim x$  forever after $\tau_x$. So the conditional Borel-Cantelli lemma ensures that only finitely many $x$'s can be of type $3$.

\vspace{0.2cm}
Now we define inductively a growing sequence of subgraphs $(G_n)_{n\ge 0}\subset G$, an increasing sequence of random times $(\tau_n)_{n\ge 0}$ and a sequence of vertices $(x_n)_{n\ge 0}$ as follows. 
First  $G_0=\{x_0\}\cup C^1_{x_0}$ and $\tau_0=0$. Note that $C^1_{x_0}=\{y\sim x_0\}$. 
Now assume that $G_k$, $\tau_k$ and $x_k$ have been defined for all $k\le n$, for some $n\ge 0$. 
If the walk never exists the set $G_n$, then set $G_{n+1}=G_n$ and $\tau_{n+1}=\infty$.
Otherwise let $\tau_{n+1}$ be the first exit time of $G_n$, $x_{n+1}$ be the position of the walk at time $\tau_{n+1}$ and set $G_{n+1}=G_n\cup \{x_{n+1}\}\cup C^1_{x_{n+1}}$.

\vspace{0.2cm}
For all $n\ge 0$ with $\tau_n<\infty$, set $O_n=\cup_{k=0}^n C^2_{x_k}\setminus G_n$, so that $\tau_{n+1}$ is the first entrance time of the walk in $O_n$.
Observe that if for some $n$, $O_n=\emptyset$, then $\tau_{n+1}=\infty$ and thus the walk will visit only a finite number of sites. Note also conversely that if the walk visits an infinite number of vertices,  then $\tau_n<\infty$, for all $n$.

\vspace{0.2cm}
Define now a filtration $(\cG_n)_{n\ge 0}$ by $\cG_n=\sigma(\xi_x\, :\  x\in G_n)$.
Then, we have that $\sigma(X_k:\; k\le \tau_{n+1})\subset \cG_n$, 
and that $x_{n+1}$ and its type are $\cG_n$ measurable.
Likewise the number $R_n$ of sites in $O_n$ is also $\cG_n$-measurable. 
Moreover, since $x_{n+1}\in O_n$, we have 
$R_{n+1}\le R_n - 1 + |C^2_{x_{n+1}}|$ (where if $A$ is a set we denote by $|A|$ its cardinality). 
Set, for $n\ge 1$, $\ell_{n}=-1$ if $x_{n}$ is of type $1$ and on $E_{x_{n}}(N)$ 
or if $x_{n}$ is of type $2$ and on $\cap_{y\in C^1_{x_{n}}} E_y(N)$, and set $\ell_{n}=d^2-1$ otherwise. Thus $\ell_{n}$ is $\cG_{n}$-measurable and for $n\ge 0$, $R_{n+1}-R_n\le -1+|C^2_{x_{n+1}}|\le \ell_{n+1}$. 
Now it is easy to see that if $x_{n+1}$ is of type $1$, 
$$\EE[\ell_{n+1} \mid \cG_n]\le -1+\epsilon d^2 < 0.$$
Similarly if $x_{n+1}$ is of type $2$, 
$$\EE[\ell_{n+1} \mid \cG_n]\le  -1+(1-(1-\epsilon)^d) d^2\le -1 + \epsilon d^3  < 0.$$
Thus, setting $A_n:={\{x_n \hbox{ is of type $1$ or $2$}\}}$, we get that $L_n :=\sum_{k=1}^n \ell_k \one_{A_k}$ is a supermartingale. 
Therefore, a.s. on the event $A_n$ occurs an infinite number of times,  
$\liminf_{n\to\infty} L_n=-\infty$ (using that on this event, $L_n$ cannot converge since on $A_n$, $|\ell_{n+1}|\ge 1$).

\vspace{0.2cm} 
Therefore a.s., if the walk visits an infinite number of vertices, then since $x_n$ is of type $3$ for only a finite number of $n$, a.s. on the event $A_n$ occurs an infinite number of times, $\liminf_{n\to\infty}R_n=-\infty$, which is absurd since, $R_n$ is nonnegative by definition. This proves the lemma.
\end{proof}

\bigskip

\paragraph*{Acknowledgements}
We acknowledge financial support from the Swiss National Foundation Grant  FN 200021-138242/1 and the french ANR project MEMEMO2 2010 BLAN 0125.
\medskip

{\footnotesize %
   \noindent Michel~\textsc{Bena\"im},
e-mail: \texttt{michel.benaim(AT)unine.ch}

\noindent\textsc{Institut de Math\'ematiques, Universit\'e de Neuch\^atel,
11 rue \'Emile Argand, 2000 Neuch\^atel, Suisse.}
\bigskip

\noindent Olivier \textsc{Raimond},
e-mail: \texttt{olivier.raimond(AT)u-paris10.fr}

 \noindent\textsc{MODAL'X Universit\'e Paris Ouest,
   200 avenue de la r\'epublique, 92000 Nanterre, France.}
 \bigskip

 \noindent Bruno \textsc{Schapira},
 e-mail: \texttt{bruno.schapira(AT)latp.univ-mrs.fr}

 \noindent\textsc{Centre de Math\'ematiques et Informatique,
Aix-Marseille Universit\'e, Technop\^ole Ch\^ateau-Gombert, $39$ rue F. Joliot Curie, 13453 Marseille, France.}
}

\end{document}